\documentclass[12pt]{amsart}
\usepackage{amsmath}
\usepackage{amssymb}
\usepackage{epsfig}
\usepackage{graphicx}
\numberwithin{equation}{section}

\usepackage[OT2,T1]{fontenc}
\DeclareSymbolFont{cyrletters}{OT2}{wncyr}{m}{n}
\DeclareMathSymbol{\Sha}{\mathalpha}{cyrletters}{"58}

\input xy
\xyoption{all}

\calclayout
\allowdisplaybreaks[3]

\theoremstyle{plain}
\newtheorem{prop}{Proposition}

\newtheorem{theo}[prop]{Theorem}
\newtheorem{coro}[prop]{Corollary}

\newtheorem{lemm}[prop]{Lemma}

\theoremstyle{definition}
\newtheorem{defi}[prop]{Definition}

\newtheorem{ques}[prop]{Question}

\newtheorem{rema}[prop]{Remark}

\newtheorem{exam}[prop]{Example}

\makeatother
\makeatletter

\newcommand{\bC}{\mathbb C}
\newcommand{\bG}{\mathbb G}

\newcommand{\bN}{\mathbb N}
\newcommand{\bP}{\mathbb P}
\newcommand{\bQ}{\mathbb Q}
\newcommand{\bR}{\mathbb R}
\newcommand{\bZ}{\mathbb Z}

\newcommand{\cC}{\mathcal C}
\newcommand{\cD}{\mathcal D}
\newcommand{\cE}{\mathcal E}
\newcommand{\cF}{\mathcal F}

\newcommand{\cM}{\mathcal M}
\newcommand{\cO}{\mathcal O}
\newcommand{\cP}{\mathcal P}
\newcommand{\cS}{\mathcal S}

\newcommand{\cX}{\mathcal X}
\newcommand{\cY}{\mathcal Y}
\newcommand{\cZ}{\mathcal Z}

\newcommand{\wcX}{\widetilde{ \mathcal X}}

\newcommand{\Br}{\operatorname{Br}}
\newcommand{\ch}{\operatorname{ch}}
\newcommand{\Def}{\operatorname{Def}}

\newcommand{\Gal}{\operatorname{Gal}}
\newcommand{\Gr}{\operatorname{Gr}}
\newcommand{\Isom}{\operatorname{Isom}}
\newcommand{\Pic}{\operatorname{Pic}}
\newcommand{\CH}{\operatorname{CH}}
\newcommand{\SL}{\operatorname{SL}}
\newcommand{\Spec}{\operatorname{Spec}}
\newcommand{\Aut}{\operatorname{Aut}}
\newcommand{\Ext}{\operatorname{Ext}}

\newcommand{\Td}{\operatorname{Td}}
\newcommand{\indd}{\operatorname{inddec}}
\newcommand{\indx}{\operatorname{ind}}

\newcommand{\ra}{\rightarrow}

\newcommand{\bfr}{\mathbf r}

\author{Brendan Hassett}
\address{Department of Mathematics\\
Rice University, MS 136 \\
Houston, TX  77251-1892 \\
USA}
\curraddr{Department of Mathematics\\
Brown University \\
Box 1917 \\
151 Thayer Street
Providence, RI 02912 \\
USA}
\email{bhassett@math.brown.edu}

\author{Yuri Tschinkel}
\address{Courant Institute\\
                New York University \\
                New York, NY 10012 \\
                USA }
\email{tschinkel@cims.nyu.edu}

\address{Simons Foundation\\
160 Fifth Avenue\\
New York, NY 10010\\
USA}

\title{Rational points on K3 surfaces and derived equivalence}

\begin{document}
\date{\today}

\maketitle

The geometry of vector bundles and derived categories on complex
K3 surfaces has developed rapidly since Mukai's seminal work \cite{Mukai}.
Many foundational questions have been answered:
\begin{itemize}
\item{the existence of vector bundles and twisted sheaves with prescribed
invariants;}
\item{geometric interpretations of isogenies between K3 surfaces \cite{Orlov,CalThe};}
\item{the global Torelli theorem for holomorphic symplectic manifolds \cite{Verb,HuySur};}
\item{the analysis of stability conditions and its implications for birational geometry
of moduli spaces of vector bundles and more general objects in the derived category \cite{BM1,BM2,Br1}.}
\end{itemize}
Given the precision and power of these results, it is natural to seek arithmetic applications of this
circle of ideas.  Questions about zero cycles on K3 surfaces have attracted the attention of 
Beauville-Voisin \cite{BV}, Huybrechts \cite{HuyMSRI}, and other authors.

Our focus in this note is on {\em rational points} over non-closed fields of arithmetic
interest. 
We seek to relate the notion of derived equivalence to arithmetic problems
over various fields.  Our guiding questions are:
\begin{ques} \label{ques:one}
Let $X$ and $Y$ be K3 surfaces, derived equivalent over a field $F$.  Does the
existence/density of rational points of $X$ imply the same for $Y$?
\end{ques}
Given $\alpha \in \Br(X)$, let $(X,\alpha)$ denote the twisted K3 surface
associated with $\alpha$, i.e., if $\cP \ra X$ is an \'etale projective bundle representing
$\alpha$, of relative dimension $r-1$ then 
$(X,\alpha)=[\cP/\SL_r].$
\begin{ques} \label{ques:two}
Suppose that $(X,\alpha)$ and $(Y,\beta)$ are derived equivalent
over $F$.  Does the existence of a rational point on the former imply the same 
for the latter?
\end{ques}
Note that an $F$-rational point of $(X,\alpha)$ corresponds to an $x\in X(F)$ such 
that $\alpha|x=0\in \Br(F)$. After this paper was released, 
Ascher, Dasaratha, Perry, and Zhou \cite{ADPZ} found that
Question~\ref{ques:two} has a negative answer, even over local fields.

We shall consider these questions for $F$ finite, $p$-adic, real, and local with algebraically closed
residue field.  These will serve as a foundation for studying how the geometry of K3 surfaces
interacts with Diophantine questions over local and global fields.  For instance, is
the Hasse/Brauer-Manin formalism over global fields compatible with (twisted) derived equivalence?
See \cite{HVA1,HVA2,MSTVA} for concrete applications to rational points problems.

In this paper, we first review general 
properties of derived equivalence over arbitrary base fields.  We then offer
examples which illuminate some of the challenges in applying derived category techniques.   
The case of finite and real fields is presented first---here the picture is well developed.
Local fields of equicharacteristic zero are also fairly well understood, at least for
K3 surfaces with semistable or other mild reduction.  The analogous questions in mixed
characteristic 
remain largely open, but comparison with the geometric case suggests a number of
avenues for future investigation.

\

{\bf Acknowledgments}:  We are grateful to Jean-Louis
Colliot-Th\'el\`ene, Daniel Huybrechts, Christian Liedtke, Johannes Nicaise,
Sho Tanimoto,
Anthony V\'arilly-Alvarado, Olivier Wittenberg,
and Letao Zhang for helpful conversations.  
We thank Vivek Shende for pointing out the application of A'Campo's Theorem.
We have benefited enormously from the thoughtful feedback of the referees.
The first author was supported by NSF
grants 0901645, 0968349, and 1148609; the second
author was supported by NSF grants
0968318 and 1160859.  We are grateful to the American Institute of 
Mathematics for sponsoring workshops where these ideas were explored.

\section{Generalities on derived equivalence for K3 surfaces}
\subsection{Definitions}
Let $X$ and $Y$ denote
K3 surfaces over a field $F$.
Let $p$ and $q$ be the projections from $X\times Y$ to $X$ and $Y$ respectively.

Let $\cE \in D^b(X\times Y)$ be an object of the bounded derived category of coherent sheaves,
which may be represented by a perfect complex of locally free sheaves.
The {\em Fourier-Mukai transform} is defined
$$\begin{array}{rcl}
\Phi_{\cE}:D^b(X) & \ra & D^b(Y) \\
           \cF & \mapsto & q_* (\cE \otimes p^*\cF),
\end{array}$$
where push forward and tensor product are the derived operations.

We say $X$ and $Y$ are {\em derived equivalent} if there exists an equivalence of
triangulated categories over $F$
$$\Phi:D^b(X) \stackrel{\sim}{\ra} D^b(Y).$$
A fundamental theorem of Orlov \cite{Orlov} implies that $\Phi$ arises 
as the Fourier-Mukai transform $\Phi_{\cE}$ for some perfect complex
$\cE \in D^b(X\times Y)$.
When we refer to Fourier-Mukai transforms below they will 
induce derived equivalences.

\subsection{Mukai lattices}

Suppose $X$ is a K3 surface defined over $\bC$ and consider its Mukai lattice
$$
\tilde{H}(X,\bZ)=\tilde{H}(X,\bZ):=H^0(X,\bZ)(-1) \oplus H^2(X,\bZ) \oplus H^4(X,\bZ)(1),
$$
where we apply Tate twists to get a Hodge structure
of weight two.  Mukai vectors refer to type $(1,1)$ vectors in $\tilde{H}(X,\bZ)$. 
Let $\left(,\right)$ denote the natural nondegenerate pairing
on $\tilde{H}(X,\bZ)$; this coincides with the intersection pairing on $H^2(X,\bZ)$
and the negative of the intersection pairing on the other summands.    
There is an induced homomorphism on the level of cohomology:
$$\begin{array}{rcl}
\phi_{\cE}:\tilde{H}(X,\bZ) & \ra & \tilde{H}(Y,\bZ) \\
   \eta & \mapsto & q_*(\sqrt{\Td_{X\times Y}}\cdot \ch(\cE) \cup p^*\eta ).
\end{array}
$$
Note that
$$\phi_{\cE}\ch(\cF)=\ch(\Phi_{\cE}(\cF)).$$

For $X$ defined over a non-closed field $F$,
there are analogous constructions in $\ell$-adic and other flavors of cohomology \cite[\S~2]{LO}.
When working $\ell$-adically,
we interpret $\tilde{H}(X,\bZ_{\ell})$ as a Galois representation rather
than a Hodge structure.  
Observe that $\phi_{\cE}$ is 
defined on Hodge structures, de Rham cohomologies, and
$\ell$-adic cohomologies---and these are all compatible.

\subsection{Characterizations over the complex numbers}
\begin{theo} \cite[\S 3]{Orlov}  \label{theo:orlov}
Let $X$ and $Y$ be K3 surfaces over $\bC$, with transcendental
cohomology groups
$$T(X):=\Pic(X)^{\perp} \subset H^2(X,\bZ), \quad
T(Y):=\Pic(Y)^{\perp} \subset H^2(Y,\bZ).$$
The following are equivalent
\begin{itemize}
\item{there exists an isometry of Hodge structures $\tilde{H}(X,\bZ)
\simeq \tilde{H}(Y,\bZ)$;}
\item{there exists an isometry of Hodge structures $T(X) \simeq T(Y)$;}
\item{$X$ and $Y$ are derived equivalent;}
\item{$Y$ is isomorphic to a moduli space of stable vector bundles over $X$,
admitting a universal family $\cE \ra X \times Y$, i.e., $Y=M_v(X)$ for a 
Mukai vector $v$ such that there exists a Mukai vector $w$ with $\left(v,w\right)=1$.}
\end{itemize}
\end{theo}
See \cite[\S 3.8]{Orlov} and \cite[\S 3]{HuyJAG08}
for discussion of the fourth condition; for our purposes, we 
need not distinguish among various notions of stability.
This has been extended to arbitrary fields as follows:
\begin{theo} \cite[Th.~1.1]{LO} \label{theo:LO}
Let $X$ and $Y$ be K3 surfaces over an algebraically closed
field $F$ of characteristic $\neq 2$.  Then the third and the fourth
statements are equivalent.
\end{theo}
Morover, a derived equivalence between $X$ and $Y$ induces Galois-compatible 
isomorphisms between their $\ell$-adic Mukai lattices.
See \cite[\S 16.4]{HuyK3} for more discussion.

\subsection{Descending derived equivalence}
Let $F$ be a field of characteristic zero with algebraic closure $\bar{F}$.
Given $Y$ smooth and projective 
over $F$, let
$\bar{Y}$ denote the corresponding variety
over $\bar{F}$.
We say $Y$ is {\em of K3 type} if $\bar{Y}$ is deformation 
equivalent to the Hilbert scheme of a K3 surface.

\begin{lemm} \label{lemm:HKdescend}
Suppose $Y_1$ and $Y_2$ are of K3 type over $F$ and there exists an isomorphism
$\iota: \bar{Y}_1 \stackrel{\sim}{\ra} \bar{Y}_2$ defined over $\bar{F}$
inducing
$$\iota^*: H^2(\bar{Y}_2,\bZ_{\ell}) \stackrel{\sim}{\ra} H^2(\bar{Y}_1,\bZ_{\ell})$$
compatible with the action of $\Gal(\bar{F}/F)$. Then $Y_1 \simeq Y_2$ over $F$.
\end{lemm}
\begin{proof}
In characteristic zero,
Galois fixed points of $\Isom(\bar{Y}_1,\bar{Y}_2)$ 
correspond to isomorphisms between the varieties defined over $F$.
Automorphisms of $\bar{Y}_1$ (resp.~$\bar{Y}_2$) act transitively and faithfully
on this set by pre-composition (resp.~post-composition).

The Torelli theorem implies that the automorphism group of a K3 surface
has a faithful representation in its second cohomology.  This holds true for manifolds
of K3 type as well---see \cite[Prop.~1.9]{Markman2010} as well as previous
work of Beauville and Kaledin-Verbitsky.
The Galois invariance of $\iota^*$ implies that
$\iota$ is Galois-fixed, hence defined over $F$.
\end{proof}

Suppose that $X$ and $Y$ are K3 surfaces over $F$.
Given a derived equivalence 
$$\Phi:D(X) \stackrel{\sim}{\ra} D(Y)$$
over $F$ the induced
$$\phi: \tilde{H}(\bar{X},\bZ_{\ell})\stackrel{\sim}{\ra} \tilde{H}(\bar{Y},\bZ_{\ell})$$
is compatible $\Gal(\bar{F}/F)$ actions. We consider whether the
converse holds.

We use the notation $\cM_v(X)$ for the moduli {\em stack} of vector
bundles/complexes over $X$; this is typically a $\bG_m$-gerbe over
$M_v(X)$ due to homotheties.
\begin{prop} \label{prop:descend}
Suppose that $X$ and $Y$ are K3 surfaces over $F$.
Let
$\Phi:D(\bar{X})\stackrel{\sim}{\ra} D(\bar{Y})$ be a derived
equivalence over $\bar{F}$ such that the
induced 
$$\phi: \tilde{H}(\bar{X},\bZ_{\ell})\stackrel{\sim}{\ra} \tilde{H}(\bar{Y},\bZ_{\ell})$$
is compatible with Galois actions. Write $v=\phi^{-1}(1,0,0)$
and $w=\phi(1,0,0)$. 

Then we have isomorphisms $Y\simeq M_v(X)$ and $X\simeq M_w(Y)$
inducing the derived equivlance over $\bar{F}$.
Moreover, there exist $\bG_m$-gerbes $\cX\ra X$ and $\cY\ra Y$
and isomorphisms $\cY\simeq \cM_v(X)$ and $\cX\simeq \cM_w(Y)$ such that
$\cX\ra X$ and $\cY \ra Y$ admit trivializations over $\bar{F}$.
\end{prop}
\begin{proof}
Choosing an appropriate polarization on $X$, the moduli space 
$M_v({\bar X})$ is a K3 surface. This follows from \cite[Prop.~4.2]{Mukai};
Mukai's argument shows that the requisite polarization can be found
over a non-closed field, as we just have to avoid walls orthogonal
to certain Mukai vectors.

There exists a Mukai vector $v'\in \tilde{H}(\bar{X},\bZ_{\ell})$
that is Galois invariant and satisfies $\left(v,v'\right)=1$, e.g.,
$v'=\phi^{-1}(0,0,-1)$. Thus the universal sheaf
$$\cE \ra X \times \cM_v(X)$$
has the following property: on basechange to $\bar{F}$ 
it may be obtained as the pull-back of a universal
sheaf on the coarse space
$$\bar{\cE} \ra \bar{X} \times M_v(\bar{X}).$$
Indeed, the universal sheaf exists wherever $M_{v'}(X)$ admits a
rational point
(see ~\cite[Th.~3.16]{LO}).

We apply Lemma~\ref{lemm:HKdescend} to $Y$ and $M_v(X)$. The Torelli Theorem
implies $\bar{Y}\simeq M_v(\bar{X})$ with the induced isomorphism
compatible with the Galois actions on $\ell$-adic cohomology.
Thus it descends to give $Y\simeq M_v(X)$ defined over $F$.
\end{proof}

\begin{rema}
Can $X$ and $Y$ fail to be derived equivalent
over $F$? How do the pull-back homomorphisms
$$\Br(F) \ra \Br(X), \quad \Br(F) \ra \Br(Y)$$
compare? They must have the same kernel if $X$ and $Y$ are derived
equivalent over $F$.
\end{rema}

\begin{rema}
Sosna \cite{Sosna} has given examples of complex K3 surfaces
that are derived equivalent but not isomorphic to their
complex conjugates.
\end{rema}

\subsection{Cycle-theoretic invariants of derived equivalence}
\begin{prop} \label{prop:PicBr}
Let $X$ and $Y$ be derived equivalent K3 surfaces over a field $F$
of characteristic $\neq 2$.
Then $\Pic(X)$ and $\Pic(Y)$ are stably
isomorphic as $\Gal(\bar{F}/F)$-modules,
and $\Br(X)[n]\simeq \Br(Y)[n]$
provided $n$ is not divisible by the characteristic.  
\end{prop}
Even over $\bC$, this result does not extend to higher dimensional
varieties \cite{Adding}.
\begin{proof}
The statement on the Picard groups follows from the Chow realization
of the Fourier-Mukai transform---see \cite[\S 2.7]{LO} for discussion.
We write
$$\tilde{H}_{\text{\'et}}(X,\mu_n)=H^0_{\text{\'et}}(X,\bZ/n\bZ) \oplus
H^2_{\text{\'et}}(X,\mu_n) \oplus
H^4_{\text{\'et}}(X,\mu_n^{\otimes 2}).$$ 
The Chow-theoretic interpretation of the Fourier-Mukai kernel gives 
the realization
$$\phi:\tilde{H}_{\text{\'et}}(X,\mu_n) \ra \tilde{H}_{\text{\'et}}(Y,\mu_n),$$
compatible with cycle class maps \cite[Prop.~2.10]{LO}.
Modding out by the images
of the cycle class maps we get the desired equality of Brauer groups.
\end{proof}

Recall that the {\em index} $\indx(X)$ of a smooth projective variety $X$ over a field $F$
is the greatest common divisor of the degrees of field extensions $F'/F$ over which
$X(F')\neq \emptyset$, or equivalently, the lengths of zero-dimensional
subschemes $Z\subset X$.

Given a bounded complex of locally free sheaves on $X$
$$E=\{E_{-M} \ra E_{-M+1} \ra \cdots \ra E_N \}$$
we may define the Chern character
$$\ch(E)=\sum_j (-1)^j \ch(E_j)$$
in Chow groups with $\bQ$ coefficients.  
The degree zero and one pieces yield the rank and first 
Chern class of $E$, expressed as alternating sums
of the ranks and determinants of the terms, respectively.  
Similarly, we may define
$$c_2(E)=-\ch_2(E)+\ch_1(E)^2/2,$$
a quadratic expression in the Chern classes of the $E_j$ 
with {\em integer} coefficients.  
Modulo the $\bZ$ algebra generated by the first Chern classes of
the $E_j$, we may write
$$c_2(E)\equiv \sum_j (-1)^j c_2(E_j).$$ 

\begin{lemm} \label{lemm:c2}
If $(S,h)$ is a smooth projective surface over $F$ then
$$\begin{array}{rcl}
\indx(S)&=&\gcd \{c_2(E): E \text{ vector bundle on } S\} \\
&=&\gcd \{c_2(E): E \in D^b(S) \}.
\end{array}
$$
\end{lemm}
\begin{proof}
Consider the `decomposible index'
$$\indd(S):=\gcd \{ D_1\cdot D_2: D_1,D_2 \text{ very ample divisors on } S\}$$
which is equal to
$$\gcd \{ D_1\cdot D_2: D_1,D_2 \text{ divisors on } S\},$$
because for any divisor $D$ the divisor $D+Nh$ is very ample for $N \gg 0$.
All three quantities in the assertion divide $\indd(S)$, so we work modulo this
quantity.

By the analysis of Chern classes above, the second and third quantities agree.  
Given a reduced zero-dimensional
subscheme $Z\subset S$ we have a resolution
$$0 \ra E_{-2} \ra E_{-1} \ra \cO_S \ra \cO_Z \ra 0$$
with $E_{-2}$ and $E_{-1}$ vector bundles.  This implies
that 
$$\gcd \{c_2(E): E \text{ vector bundle on } S\} | \indx(S).$$
Conversely, given a vector bundle $E$ there exists a twist
$E\otimes \cO_S(Nh)$ that is globally generated and
$$c_2(E\otimes \cO_S(Nh))\equiv c_2(E) \pmod{\indd(S)}.$$
Thus there exists a zero-cycle $Z$ with
degree $c_2(E\otimes \cO_S(Nh))$ and
$$\indx(S) | \gcd \{c_2(E): E \text{ vector bundle on } S\}.$$
\end{proof}

\begin{prop}
If $X$ is a K3 surface over a field $F$ then
$$\indx(X) | \gcd \{ 24, D_1\cdot D_2 \text{ where } D_1,D_2 \text{ are divisors on }X \}.$$
\end{prop}
This follows from Lemma~\ref{lemm:c2} and
the fact that $c_2(T_X)=24$.  Beauville-Voisin \cite{BV} and Huybrechts \cite{HuyEMS} have
studied the corresponding subgroup of $\CH_0(X_{\bar{F}})$.  

\begin{prop} \label{prop:index}
Let $X$ and $Y$ be derived equivalent K3 surfaces over a field $F$.
Then $\indx(X)=\indx(Y)$.
\end{prop}
The proof will require the notion of a {\em spherical object} on a K3
surface. This is an object
$\cS \in D^b(X)$ with
$$\Ext^0(\cS,\cS)=\Ext^2(\cS,\cS)=F,\quad \Ext^i(\cS,\cS)=0, \quad i\neq 0,2.$$
These satisfy the following
\begin{itemize}
\item{$\left(v(\cS),v(\cS)\right)=-2$;}
\item{rigid simple vector bundles are spherical;}
\item{each spherical object $\cS$ has the associated
spherical twist \cite[16.2.4]{HuyK3}:
$$T_{\cS}: D^b(X) \ra D^b(X),$$
an autoequivalence such that the induced homomorphism on the Mukai lattice
is the reflection associated with $v(\cS)$;}
\item{each spherical object $\bar{\cS}$ on $X_{\bar{F}}$ is defined
over a finite extension $F'/F$ \cite[5.4]{HuyEMS};}
\item{over $\bC$, each $v=(r,D,s)\in \tilde{H}(X,\bZ) \cap H^{1,1}$ with
$\left(v,v\right)=-2$ arises from a spherical object, which
may be taken to be a rigid vector bundle $E$ if $r>0$ \cite{Kuleshov};}
\item{under the same assumptions, for each polarization $h$ on $X$ there is
a {\em unique} $h$-slope stable vector bundle $E$ with $v(E)=v$ \cite[5.1.iii]{HuyMSRI}.}
\end{itemize}
\begin{proof}
Proposition~\ref{prop:descend} allows us to
express $Y=M_v(X)$ for $v=(r,ah,s)$ where $h$ is a polarization on
$X$ and $a^2h^2=2rs$.  It also yields
a Mukai vector $w=(r',bg,s')\in \tilde{H}(X,\bZ_{\ell})$
with 
$$\left(v,w\right)=ab g\cdot h -r s' - s r'=1.$$ 
Thus we have
\begin{equation} \label{gcdrs}
\left<r,s\right>=\left<1\right> \pmod{g\cdot h}.
\end{equation}

Consider a Fourier-Mukai transform realizing the equivalence
$$\Phi:D^b(X) \ra D^b(Y)$$
and the induced homomorphism $\phi$ on the Mukai lattice.  Note that
$$\phi(v)=(0,0,1)$$ 
reflecting the fact that a point on $Y$ corresponds to a sheaf on $X$
with Mukai vector $v$.

Suppose that $Y$ has a rational point over a field of degree $n$
over $F$; let $Z \subset Y$ denote the corresponding subscheme of
length $n$.  Applying $\Phi^{-1}$ to $\cO_Z$ gives an element of the
derived category with Mukai vector
$(nr,nah,ns)$ and
$$c_2(\Phi^{-1}(\cO_Z))=\frac{c_1(\Phi^{-1}(\cO_Z))^2}{2}-\chi(\Phi^{-1}(\cO_Z))+
2\operatorname{rank}(\Phi^{-1}(\cO_Z))$$
which equals
$n(nrs+r-s).$
Following the proof of Lemma~\ref{lemm:c2}, we compute 
$$
c_2(\Phi^{-1}(\cO_Z)) \pmod{\indd(X)}.
$$
First suppose that $r$ and $s$ have different parity, so that
$$
\gcd(nrs+r-s,2rs)=\gcd(nrs+r-s,rs).
$$
Then using (\ref{gcdrs}) we obtain
\begin{align*}
\left<nrs+r-s,rs\right>=\left<r-s,rs\right>=\left<r-s,r\right>\left<r-s,s\right>\\
=\left<-s,r\right>\left<r,s\right>=\left<r,s\right>^2=\left<1\right> \pmod{g\cdot h}.
\end{align*}
If $r$ and $s$ are both even then $g\cdot h$ must be odd
and 
$$\left<nrs+r-s,2rs\right>=\left<nrs+r-s,rs\right> \pmod{g\cdot h}$$
and repeating the argument above gives the desired conclusion.

Now suppose that $r$ and $s$ are both odd.
It follows that $h^2\equiv 2\pmod{4}$ and we write $h^2=2\gamma-2$
for some even integer $\gamma$. Let $\cS$ denote the spherical object
associated with $h$ so that
$v(\cS)=(1,h,\gamma)$. Applying $T_{\cS}$ to the Mukai vector
$$(r,ah,s) \mapsto (r,ah,s) + \left((r,ah,s),(1,h,\gamma)\right)(1,h,\gamma),$$
we obtain a new vector with rank $r+(ah^2-s-r\gamma)$, which is even.
This reduces us to the previous situation.

In each case, we find
$$
c_2(\Phi^{-1}(\cO_Z))\equiv n \pmod{\indd(X)},
$$
whence $\indx(X)|n$.  Varying over all degrees $n$, we find
$$\indx(X)|\indx(Y)$$
and the Proposition follows.
\end{proof}

The last result raises the question of whether spherical objects are defined
over the ground field:
\begin{ques}
Let $X$ be a K3 surface over a field $F$.
Suppose that $\bar{\cS}$ is a spherical object on $X_{\bar{F}}$ such
that $c_1(\bar{\cS}) \in \Pic(X_{\bar{F}})$ is a divisor defined
over $X$.  When does $\bar{\cS}$ come from an object $\cS$ on $X$?
\end{ques}

Kuleshov \cite{Kuleshov,Kuleshov2} gives a partial description
of how to generate all exceptional bundles on K3 surfaces of Picard rank 
one through `restructuring' operations and `dragons'.  It would be
worthwhile to analyze which of these operations could be
defined over the ground field.

\begin{exam} \label{exam:14}
We give an example of a K3 surface $X$ over a field $F$ with 
$$\Pic(X)=\Pic(X_{\bar{F}})=\bZ h$$
and a rigid sheaf $E$ over $X_{\bar{F}}$ that fails to descend 
to $F$.  

Choose $(X,h)$ to be a degree fourteen K3 surface defined over $\bR$
with $X(\bR)=\emptyset$.  This may be constructed as follows:  
Fix a smooth conic $C$ and quadric threefold $Q$ with
$$C \subset Q \subset \bP^4, \quad Q(\bR)=\emptyset.$$
Let $X'$ denote a complete intersection of $Q$ with a cubic
containing $C$; we have $X'(\bR)=\emptyset$ and $X'$ admits a 
lattice polarization
$$\begin{array}{c|cc}
   & g & C \\
\hline
g   & 6 & 2 \\
C  & 2 & -2
\end{array}.
$$
Write $h=2g-C$ so that $(X',h)$ is a degree $14$ K3 surface
containing a conic.  Let $X$ be a small deformation of $X'$
with $\Pic(X_{\bC})=\bZ h$.  

The K3 surface $X$ is Pfaffian if and only if it admits a vector
bundle $E$ with $v(E)=(2,h,4)$ corresponding to the classifying
morphism $X \ra \Gr(2,6)$.  However, note that
$$c_2(E)=5$$
which would mean that $\indx(X)=1$.  On the other hand,
if $X(\bR)=\emptyset$ then
$\indx(X)=2$.  
\end{exam}

\section{Examples of derived equivalence}
\label{sect:examples}

\subsection{Elliptic fibrations}
The paper \cite{AKW} has a detailed discussion of derived equivalences 
among genus one curves over function fields. 

In this section we work over a field $F$ of characteristic zero.

A K3 surface $X$ is {\em elliptic} if it admits a morphism $X \ra C$
to a curve of genus zero with fibers of genus one.  We allow 
$C=\bP^1$ or a non-split conic over $F$.

\begin{lemm} \label{lemm:whenelliptic}
A K3 surface $X$ is elliptic if and only if it admits a non-trivial
divisor $D$ with $D^2=0$.
\end{lemm}
\begin{proof}
If $X$ is elliptic then the pull back of a non-trivial divisor from $C$
gives a square-zero class; we focus on the converse.

This is well-known if $F$ is algebraically closed \cite[\S 6, Th.~1]{PSSh}.
Indeed, let $\cC_+ \subset H^2(X,\bR)$ denote the component of the
positive cone $\{\eta: \left(\eta,\eta\right)>0 \}$ containing an ample
divisor and $\Gamma \subset \operatorname{Aut}(H^2(X,\bZ))$ 
the group generated by Picard-Lefschetz reflections $\rho_R$
associated with $(-2)$-curves $R \in \Pic(X)$. Then 
the K\"ahler cone
of $X$ is a fundamental domain for the action of $\Gamma$
on $\cC_+$, in the sense that no two elements of the cone are in the 
same orbit and each orbit in $\cC_+$ meets the 
closure of the K\"ahler cone. Cohomology classes in the interior of the cone have
trivial $\Gamma$-stabilizer; classes on walls of the boundary associated
with $(-2)$-curves are stabilized by the associated
reflections.

Thus if $\Pic(X)$ represents zero
then there exists a non-zero divisor $D$ in the closure of the K\"ahler
cone with $D^2=0$, which induces an elliptic fibration $X\ra \bP^1$.
Given a divisor $D \in \cC_+$, we can be a bit more precise about
the $\gamma \in \Gamma$ required to take $D$ to a nef
divisor.  We can write
$$\gamma=\rho_{R_1} \cdots \rho_{R_m}$$
where each $R_j$ is the class of an irreducible rational curve
contained in the fixed part of the linear series $|D|$ 
(cf.~proof in \cite[\S 6]{PSSh}, \cite[\S 2]{SD}).

Now suppose $F$ is not algebraically closed and $D$ is defined over $F$.  
The group $\Gal(\bar{F}/F)$ acts on the $(-2)$-curves on $\bar{X}$
and thus on $\Gamma$ and the orbit $\Gamma\cdot D$.
The Galois action on $R_1,\ldots,R_m$ may be nontrivial. However,
the fundamental domain description guarantees a unique
$f \in \Gamma\cdot D$ in the nef cone of $X$, which is necessarily 
Galois invariant; some multiple of this divisor is defined over $F$.
This semiample divisor induces our elliptic fibration.
\end{proof}

\begin{prop}
Let $X$ be an elliptic K3 surface over $F$ and $Y$ another
K3 surface derived equivalent to $X$ over $F$.
Then $Y$ is elliptic over $F$.
\end{prop}
\begin{proof}
By Proposition~\ref{prop:PicBr}, the Picard groups of $X$ and
$Y$ are stably isomorphic as lattices
$$\Pic(X) \oplus U \simeq \Pic(Y) \oplus U, \quad 
U=\left( \begin{matrix} 0 & 1 \\
	1 & 0  \end{matrix} \right).$$
In particular, $\Pic(X)$ and $\Pic(Y)$ share the same discriminant
and $p$-adic invariants.
A rank-two indefinite
lattice represents zero if and only if its discriminant is a square.
In higher ranks, an indefinite lattice represents zero if and only if it
represents zero $p$-adically for each $p$.

Thus $\Pic(Y)$ admits a square-zero class and is elliptic by
Lemma~\ref{lemm:whenelliptic}.
\end{proof}

\

An elliptic K3 surface $J\ra C$ is {\em Jacobian} if it admits
a section $C \ra J$.  It has geometric Picard group containing
\begin{equation} \label{eqn:lattice}
\begin{array}{c|cc}
   & f & \Sigma \\
\hline
f  & 0 & 1 \\
\Sigma & 1 & -2
\end{array}
\end{equation}
where $f$ is a fiber and $\Sigma$ is the section.  
Jacobian elliptic surfaces admit numerous autoequivalences \cite[\S 5]{BrCrelle}.
Let $a,b \in \bZ$ with $a>0$ and $(a,b)=1$.  The moduli space of rank $a$ degree $b$
indecomposable vector bundles on fibers of $J\ra C$ is an elliptic fibration
with section, isomorphic to $J$ over $C$.  This reflects Atiyah's classification
of vector bundles over elliptic curves \cite[Th.~7]{Atiyah}.  The associated Fourier-Mukai
transform induces an autoequivalence of $J$ acting on $\tilde{H}(J,\bZ)$ via an element of 
$$ \left( \begin{matrix} c & a \\ d & b\end{matrix} \right) \in \SL_2(\bZ)$$
obtained as in \cite[Th. 3.2,5.3]{BrCrelle}.
This acts on the Mukai vectors 
(or their $\ell$-adic analogues) 
$$(1,0,0), (0,0,1) \in H^0(J,\bZ)(-1) \oplus H^2(J,\bZ) \oplus  H^4(J,\bZ)(1)$$
by the formula
$$(1,0,0) \mapsto  (0,af,c), \quad \quad
(0,0,1) \mapsto  (b,d(f+\Sigma),0).$$
The action is transitive on the $(-2)$-vectors in the
lattice of algebraic classes
\begin{equation} \label{eq:twosummand}
 H^0(J,\bZ)(-1) \oplus  H^4(J,\bZ)(1) \oplus \left< f,\Sigma \right> \simeq  \left( \begin{matrix} 0 & -1 \\
			 -1 & 0 \end{matrix} \right)
\oplus
\left( \begin{matrix} 0 & 1 \\
		      1 & -2 \end{matrix} \right).
\end{equation}
Thus we have established the following:
\begin{prop} \label{prop:autoeq}
Let $J\ra C$ be a Jacobian elliptic K3 surface over $F$.
Then autoequivalences of $J$ defined over $F$ induce a representation
of $\SL_2(\bZ)$ on (\ref{eq:twosummand}) as above.
\end{prop}
The point is that here autoequivalences are defined over the ground field.

Over $\bC$, the availability of autoequivalences has
strong consequences for Jacobian elliptic K3 surfaces.
For instance,
orientation preserving automorphisms of the Mukai lattice
of a complex K3 surface arise from autoequivalences
\cite[Th.~1.6]{HLOY3}. This implies that derived equivalent
Jacobian elliptic K3 surfaces are isomorphic \cite[Cor.~2.7.3]{HLOY}.
The autoequivalences are generated via spherical twists
associated with rigid objects.

As we have seen (e.g., in Example~\ref{exam:14}), rigid objects over 
non-closed fields need not descend. It is natural to ask the following 
question:

\begin{ques} \label{ques:oneJac}
Let $J_1$ and $J_2$ denote Jacobian elliptic K3 surfaces over a field $F$
of characteristic zero.  
If $J_1$ and $J_2$ are derived equivalent does it follow that $J_1 \simeq J_2$?
\end{ques}

If the geometric Picard group has rank two then the isomorphism follows
from Proposition~\ref{prop:autoeq}.

\

We recall the classical Ogg-Shafarevich theory for elliptic
fibrations, following \cite[4.4.1,5.4.5]{CalThe}:
Let $F$ be algebraically closed of characteristic zero and 
$J\ra \bP^1$ a Jacobian elliptic K3 surface.
We may interpret
$$\Sha(J/\bP^1)\simeq \Br(J)$$
and each $\alpha$ in this group may be realized by an 
elliptically fibered K3 surface
$X \ra \bP^1$ with Jacobian fibration $J\ra \bP^1$.  
If $\alpha$ has order $n$ then we have natural
exact sequences
$$0 \ra \bZ/n\bZ \ra \Br(J) \ra \Br(X) \ra 0$$
and
$$0 \ra T(X) \ra T(J) \ra \bZ/n\bZ \ra 0.$$
Note that if $Y=\Pic^c(X/\bP^1)$ then 
$[Y]=c[X] \in \Sha(J/\bP^1)$.  
If $c$ is relatively prime to $n$, so there exists an 
integer $b$ with $bc\equiv 1\pmod{n}$, then we have
$\Pic^b(Y)=X$ as well.  It follows that
$X$ and $Y$ are derived equivalent with the universal
sheaves inducing the Fourier-Mukai transform \cite[Th.~1.2]{BrCrelle}.

\begin{prop} \label{prop:ellip1}
Let $F$ be algebraically closed of characteristic zero.  
Let $\phi:X \ra \bP^1$ be an elliptic K3 surface with Jacobian
fibration $J(X) \ra \bP^1$.  Let $\alpha \in \Br(J(X))$ denote the
Brauer class associated with $[X]$ in the Tate-Shafarevich group of 
$J(X) \ra \bP^1$.  Then $X$ is derived equivalent to the pair
$(J(X),\alpha)$.  
\end{prop}
This follows from the proof of C\u{a}ld\u{a}raru's conjecture;
see \cite[1.vi]{HuSt}
as well as \cite[4.4.1]{CalThe} for the fundamental identification between the twisting
data and the Tate-Shafarevich group.

\begin{ques}  \label{ques:DEellip}
Let $X$ and $Y$ be K3 surfaces derived equivalent over a field $F$.
Suppose we have an elliptic fibration $X\ra C$ over $F$.
Does $Y$
admit a fibration $Y\ra C$ such that 
$$\begin{array}{ccccc}
X & & \stackrel{\sim}{\ra} & & \Pic^b(Y/C) \\
  & \searrow & & \swarrow & \\
  &          &C&     &  
\end{array}
$$
for some integer $b$?
\end{ques}
By symmetry we have $X\simeq \Pic^b(Y/C)$ and $Y\simeq \Pic^c(X/C)$
for $b,c \in \bZ$. 
As before, if $n=\mathrm{ord}([X])=\mathrm{ord}([Y])$ in the Tate-Shafarevich group
then $bc\equiv 1\pmod{n}$. 

A positive answer to Question~\ref{ques:DEellip} would imply:
\begin{itemize}
\item
 $X$ dominates
$Y$ over $F$ and {\em vice versa}.
\item
If $X$ and $Y$ are elliptic K3 surfaces derived equivalent over a field $F$
of characteristic zero then $X(F) \neq \emptyset$ if and only if $Y(F)\neq \emptyset$.  
\end{itemize}

\subsection{Rank one K3 surfaces}
We recall the general picture:
\begin{prop} \cite[Prop.~1.10]{Oguiso}, \cite{Stellari04}
Let $X/\bC$ be a K3 surface with $\Pic(X)=\bZ h$, where $h^2=2n$.
Then the number $m$ of isomorphism classes of
K3 surfaces $Y$ derived equivalent to $X$ is given by
$$m = 2^{\tau(n)-1}, \quad \text{where} \quad \tau(n)=\text{number of prime factors of n}.$$
\end{prop}

\begin{exam} \label{exam:twelve}
The first case where there are multiple isomorphism classes is degree twelve.  Let $(X,h)$ be such a 
K3 surface and $Y=M_{(2,h,3)}(X)$ the moduli space of stable vector bundles $E\ra X$ with
$$\operatorname{rk}(E)=2, \quad c_1(E)=h, \quad \chi(E)=2+3=5,$$
whence $c_2(E)=5$.  Note that if $Y(F)\neq \emptyset$ then $X$ admits an effective zero-cycle of
degree five and therefore a zero-cycle of degree one.  Indeed, if $E\ra X$ is a vector bundle corresponding 
to $[E] \in Y(F)$ then a generic $\sigma \in \Gamma(X,E)$ vanishes at five points on $X$.
As we vary $\sigma$, we get a four-parameter family of such cycles.  Moreover,
the cycle $h^2$ has degree twelve, relatively prime to five.  

Is $X(F)\neq \emptyset$ when $Y(F)\neq \emptyset$?
\end{exam}

\subsection{Rank two K3 surfaces}
Exhibiting pairs of non-isomorphic derived equivalent complex K3 surfaces 
of rank two is a problem on quadratic forms \cite[\S 3]{HLOY}.   
Suppose that $\Pic(X_{\bC})=\Pi_X$ and $\Pic(Y_{\bC})=\Pi_Y$ and
$X$ and $Y$ are derived equivalent.  Orlov's Theorem implies
$T(X)\simeq T(Y)$ which means that $\Pi_X$ and $\Pi_Y$ have isomorphic
discriminant groups/$p$-adic invariants.  Thus we have to 
exhibit $p$-adically equivalent rank-two even indefinite lattices that
are not equivalent over $\bZ$.  
(Proposition~\ref{prop:PicBr} asserts this is a necessary condition
over general fields of characteristic zero.)

\begin{exam}
We are grateful to Sho Tanimoto and Letao Zhang for assistance with this
example.
Consider the lattices
$$\Pi_X=\begin{array}{c|cc}  & C & f \\
			\hline
			 C & 2 &  13 \\
			 f & 13 & 12 \end{array} \quad
\Pi_Y=\begin{array}{c|cc}  & D & g \\
			\hline
			 D & 8 &  15 \\
			 g & 15 & 10 \end{array} 
$$
which both have discriminant $145$. 
Note that $\Pi_X$ represents $-2$
$$(2f-C)^2=(25C-2f)^2=-2$$
but that $\Pi_Y$ fails to represent $-2$.  

Let $X$ be a K3 surface over $F$ with split Picard group $\Pi_X$ over a field $F$.  
We assume that $C$ and $f$ are ample.
The moduli space
$$Y=M_{(2,C+f,10)}(X)$$
has Picard group
$$\begin{array}{c|cc}  & 2C & (C+f)/2 \\
			\hline
			 2C &  8 &  15 \\
			 (C+f)/2 & 15 & 10 \end{array} \simeq \Pi_Y$$
while $M_{(2,D,2)}(Y)$
has Picard group
$$\begin{array}{c|cc}  & D/2 & 2g \\
			\hline
			 D/2 & 2 &  15 \\
			 2g & 15 & 40 \end{array} \simeq \Pi_X$$
and is isomorphic to $X$.

These surfaces have the following properties:
\begin{itemize}
\item{$X$ and $Y$ admit decomposable zero cycles of degree one over $F$;}
\item{$X(F)\neq \emptyset$: the rational points arise from the smooth
rational curves with classes $2f-C$ and $25C-2f$, both of which
admit zero-cycles of odd degree and thus are $\simeq \bP^1$ over $F$;}
\item{$Y(F')$ is dense for some finite extension $F'/F$, due to the
fact that $|\Aut(Y_{\bC})|=\infty$.}
\end{itemize}
We do not know whether
\begin{itemize}
\item{$X(F')$ is dense for any finite extension $F'/F$;}
\item{$Y(F)\neq \emptyset$.}
\end{itemize}
\end{exam}

\section{Finite and real fields}
The $\ell$-adic interpretation of the Fourier-Mukai transform yields 
\begin{theo} \label{theo:finite} \cite{LO} \cite[16.4.3]{HuyK3}
Let $X$ and $Y$ be K3 surfaces derived equivalent over a finite field $F$.
Then for each finite extension $F'/F$ we have
$$|X(F')|=|(Y(F')|.$$
\end{theo}
For the case of general surfaces see \cite{Honigs}.

\

We have a similarly complete picture over the real numbers.
We review results of Nikulin \cite[\S 3]{Nik79} \cite[\S 2]{Nik08} on
real K3 surfaces.  

Let $X$ be a K3 surface over $\bR$,
$X_{\bC}$ the corresponding complex K3 surface, 
and $\varphi$ the action of the anti-holomorphic involution (complex conjugation)
of $X_{\bC}$ on $H^2(X_{\bC},\bZ)$.  
Let $\Lambda_{\pm}\subset H^2(X_{\bC},\bZ)$
denote the eigenlattices where $\varphi$ acts via $\pm 1$.  
If $D$ is a divisor on $X$ defined over $\bR$ then
$$\varphi([D])=-D;$$ 
the sign reflects the fact that complex conjugation reverses the sign
of $(1,1)$ forms.  In Galois-theoretic terms, the cycle class of a divisor
lives naturally $H^2(X_{\bC},\bZ(1))$ and twisting by $-1$ accounts
for the sign change.  
Let $\tilde{\Lambda}_{\pm}$ denote the
eigenlattices of the Mukai lattice; note that $\tilde{\Lambda}_-$
contains the degree zero and four summands.  Again, the sign
change reflects the fact that these are twisted in the Mukai 
lattice.

We introduce the key invariants:  Let $\bfr$ denote the
rank of $\Lambda_{-}$.  The discriminant groups of
$\Lambda_{\pm}$ are two-elementary groups of order $2^a$
where $a$ is a non-negative integer.  Note that $\tilde{\Lambda}_{\pm}$
have discriminant groups of the same order.  Finally, we
set
$$\delta_{\varphi}= \begin{cases} 0 \text{ if } \left(\lambda,\varphi(\lambda)\right)
				\equiv 0(\mathrm{mod}\ 2) \text{ for each }\lambda \in
					\Lambda \\
				1 \text{ otherwise. }
		\end{cases}$$
Note that $\delta_{\varphi}$ can be computed via the Mukai lattice
$$\delta_{\varphi}=  0 \text{ iff } \left(\lambda,\varphi(\lambda)\right)
				\equiv 0(\mathrm{mod}\ 2)  \text{ for each }\lambda \in
					\tilde{\Lambda},
$$
as the degree zero and four summands always give even intersections.

We observe the following:
\begin{prop}
Let $X$ and $Y$ be K3 surfaces over $\bR$, derived equivalent over $\bR$.
Then 
$$(\bfr(X),a(X),\delta_{\varphi,X})=
(\bfr(Y),a(Y),\delta_{\varphi,Y}).$$
\end{prop}
\begin{proof}
The derived equivalence induces an isomorphism
$$\tilde{H}(X_{\bC},\bZ)\simeq \tilde{H}(Y_{\bC},\bZ)$$
compatible with the conjugation actions.
Since $(\bfr,a,\delta_{\varphi})$ can be read off from the Mukai lattice,
the equality follows.
\end{proof}

The topological type of a real K3 surface is governed by
these invariants.
Let $T_g$ denote a compact orientable surface of genus $g$.
\begin{prop} \cite[Th.~3.10.6]{Nik79} \cite[2.2]{Nik08}
Let $X$ be a real K3 surface with invariants $(\bfr,a,\delta_{\varphi})$.
Then the manifold $X(\bR)$ is orientable and 
$$ X(\bR) =
\begin{cases}
\emptyset & \text{ if } (\bfr,a,\delta_{\varphi})=(10,10,0) \\
T_1 \sqcup T_1  & \text{ if } (\bfr,a,\delta_{\varphi})=(10,8,0) \\
T_g \sqcup (T_0)^k & \text{ otherwise, where } \\
                   & \quad \quad  g=(22-\bfr-a)/2, k=(\bfr-a)/2.
\end{cases}
$$
\end{prop}

\begin{coro}
Let $X$ and $Y$ be K3 surfaces defined
and derived equivalent over $\bR$.  
Then $X(\bR)$ and $Y(\bR)$ are diffeomorphic.  
In particular, $X(\bR)\neq \emptyset$ if and only if $Y(\bR)\neq \emptyset$.
\end{coro}
The last statement also follows from Proposition~\ref{prop:index}:
A variety over $\bR$ has a real point if and only if its index is one.  
(This was pointed out to us by Colliot-Th\'el\`ene.)

\begin{exam}
Let $X$ and $Y$ be derived equivalent K3 surfaces, defined over $\bR$;
assume they have Picard rank one.  
Then
$Y=M_v(X)$ for some isotropic Mukai vector $v=(r,f,s)\in \tilde{H}(X(\bC),\bZ)$
with $(r,s)=1$.  For a vector bundle $E$ of this type note that
$$c_2(E)=c_1(E)^2/2+r\chi(\cO_X)-\chi(E)=rs+r-s,$$
which is odd as $r$ and $s$ are not both even.  Then a global section of $E$
gives an odd-degree cycle on $X$ over $\bR$, hence an $\bR$-point.
\end{exam}

\

\section{Geometric case: local fields with complex residue field}
We start with a general definition: Let $R$ be a discrete
valuation ring with quotient field $F$.
A K3 surface $X$ over $F$ has
{\em good reduction} if there exists a smooth proper algebraic space
$\cX \ra \Spec(R)$ with generic fiber $X$.  
It has {\em ADE reduction} if the central fiber has (at worst) rational
double points.  

Let $X$ be a projective K3 surface over $F=\bC((t))$.
Consider the monodromy action
$$T:H^2(X,\bZ) \ra H^2(X,\bZ)$$
associated with a loop about $t=0$.
This is {\em quasi-unipotent}, i.e., there exist
$e,f \in \bN$ such that
$(T^e-I)^f=0$ \cite{Landman}; we choose $e,f$ minimal with this property.

Let $\Delta=\Spec \bC[[t]]$ and fix a projective  model
$$\cX \ra \Delta$$
and a resolution
$$\varpi:\cX' \ra \Delta$$
such that the central fiber $\varpi^{-1}(0)$ is a normal crossings divisor, perhaps
with multiplicities along some components.  Let $\cX'_{\circ} \subset \varpi^{-1}(0)$
denote the smooth locus of the central fiber, i.e., the points of mulitiplicity one.
A'Campo \cite[Th.~1]{AC} proved that
$$\chi(\cX'_{\circ})=2+\operatorname{trace}(T).$$
This may be interpreted as the alternating sum
of the traces of the monodromy matrices  on {\em all} the cohomology groups of $X$.

An application of Hensel's Lemma yields
\begin{prop}  \label{prop:ACampo}
If
$\operatorname{trace}(T)\neq -2$ then $\cX \ra \Delta$ admits
a section, i.e., $X(F)\neq \emptyset$.
\end{prop}
This result was previously obtained by Nicaise \cite[Cor.~6.6]{Nicaise};
his techniques are also applicable in mixed characteristic under appropriate
tameness assumptions.

Proposition~\ref{prop:ACampo} applies when $T$ is unipotent ($e=1$).
This is the case when there exists a resolution
$\cX'\ra \Delta$ with central fiber {\em reduced} normal crossings.
(See the Appendix and Theorem~\ref{theo:Kulikov} for further analysis.)  

If $X$ and $Y$ are derived equivalent over $F$ then Orlov's Theorem
and the discussion preceding Proposition~\ref{prop:descend} implies that
their Mukai lattices admit a monodromy equivariant isomorphism.
Thus the characteristic polynomials of their monodromy matrices are equal.
\begin{coro}
Suppose that $X$ and $Y$ are derived equivalent K3 surfaces
with monodromy satisfying
$\operatorname{trace}(T)\neq -2$.  
Then both $X(F)$ and $Y(F)$ are nonempty.
\end{coro}

We also have results when $T$ is semisimple:
\begin{prop} \label{prop:ADE}
Suppose that $X$ and $Y$ are derived equivalent K3 surfaces
over $F=\bC((t))$.  Then the following conditions are equivalent:
\begin{itemize}
\item{$X$ (equivalently, $Y$)
has monodromy acting via an element of a product of Weyl groups;}
\item{$X$ and $Y$ admit models with central fiber consisting of a K3
surface with ADE singularities.}
\end{itemize}
Thus both $X(F)$ and $Y(F)$ are nonempty.
\end{prop}
\begin{proof}  
We elaborate on the first condition:  Let $T$ denote
the monodromy of $X$.  Then there exist vanishing
cycles $\gamma_1,\ldots,\gamma_s$ for $X$ such that each
$\gamma_i^2=-2$, $\left<\gamma_1,\ldots,\gamma_s\right>$ is
negative definite, and $T$ is a product of reflections 
associated with the $\gamma_i$.  If $g$ is any polarization 
on $X$ then the $\gamma_i$ are orthogonal to $g$.  
Thus the Fourier-Mukai transform restricts to an isomorphism on
the sublattice generated by the $\gamma_i$.  
In particular,
the monodromy of $Y$ admits the same interpretation as a product
of reflections.  

Let $L$ denote the smallest saturated sublattice of $H^2(X,\bZ)$ 
containing $\gamma_1,\ldots,\gamma_s$---the classification
of Dynkin diagrams implies it is a direct sum of
lattices of ADE type.  Let $M$ denote the corresponding lattice
in $H^2(Y,\bZ)$, which is isomorphic to $L$.

After a basechange 
$$\Spec(R_1) \ra \Spec(R), \quad t_1^e=t$$
where $e$ is the order of $T$, the Torelli Theorem gives smooth 
(Type I Kulikov) models
$$\cX_1,\cY_1 \ra \Spec(R_1)$$
with central fibers having ADE configurations of type $M$, consisting of
smooth rational curves.  Blowing these down yield models
$$\cX_1',\cY_1' \ra \Spec(R_1)$$
which descend to 
$$\cX,\cY \ra \Spec(R),$$
i.e., ADE models of $X$ and $Y$.  

An application of Hensel's Lemma gives that $X(F),Y(F)\neq \emptyset$.
\end{proof}

\section{Semistable models over $p$-adic fields}
Let $F$ be a $p$-adic field with ring of integers $R$. 

We start with the case of good reduction, which follows
from Theorem~\ref{theo:finite} and Hensel's Lemma:
\begin{coro}
Let $X$ and $Y$ be K3 surfaces over $F$, which are derived equivalent
and have good reduction 
over $F$.
Then $X(F)\neq \emptyset$ if and only if $Y(F)\neq \emptyset$.  
\end{coro}

We can extend this as follows:
\begin{prop}
Assume that the residue characteristic $p\ge 7$.  
Let $X$ and $Y$ be K3 surfaces over $F$, which are derived equivalent 
and have ADE reduction 
over $F$.  Then $X(F)\neq \emptyset$ if and
only if $Y(F)\neq \emptyset$.   
\end{prop}
\begin{proof}
Let $k$ be the finite residue field,
$\cX,\cY \ra \Spec(R)$ proper models of $X$ and $Y$,
$\cX_0$ and $\cY_0$ denote the resulting reductions,
and 
$\tilde{\cX}_0$ and $\tilde{\cY}_0$ their minimal resolutions
over $\bar{k}$.
Applying Artin's version of Brieskorn simultaneous resolution
\cite[Th.~2]{Artin}, there exists a finite extension
$$\Spec(R_1) \ra \Spec(R)$$
and proper models 
$$\tilde{\cX}\ra \cX \times_{\Spec(R)}\Spec(R')  \ra \Spec(R'), 
$$
$$
\tilde{\cY}\ra \cY \times_{\Spec(R)}\Spec(R')  \ra \Spec(R'),$$
in the category of algebraic spaces, with central fibers
$\tilde{\cX}_0$ and $\tilde{\cY}_0$.  

The Fourier-Mukai transform induces isomorphisms
$$H^2_{\text{\'et}}(\bar{X},\bQ_{\ell}) \ra
H^2_{\text{\'et}}(\bar{Y},\bQ_{\ell}).$$
Indeed, there is a natural choice on the integral transcendental cohomology.
The stable equivalence on Picard groups (see Proposition~\ref{prop:PicBr})
guarantees the existence of a isomorphism on the algebraic classes after
tensoring by $\bQ$.

Specializing yields an isomorphism 
$$\psi:H^2_{\text{\'et}}(\tilde{\cX}_0,\bQ_{\ell}) \ra
H^2_{\text{\'et}}(\tilde{\cY}_0,\bQ_{\ell}).$$
Note that since $\tilde{\cX},\tilde{\cY}$ are not projective
over $\Spec(R')$, there is not an evident interpretation of this
as a derived equivalence over $\Spec(R')$.  (See \cite{BrMa}
for such interpretations for K3 fibrations over complex curves.) 
Furthermore $\psi$ is far from unique, as we may
compose with reflections arising from exceptional curves
in either $\tilde{\cX}_0 \ra \cX_0$ or 
$\tilde{\cY}_0 \ra \cY_0$ associated with vanishing cycles
of $\cX$ or $\cY$.  

Let $L$ (resp.~$M$) denote the lattice of vanishing cycles 
in $H^2(X,\bQ_{\ell})$ (resp.~$H^2(Y,\bQ_{\ell})$),
with orthogonal complement $L^{\perp}$ (resp.$M^{\perp}$).  
The isomorphism $\psi$ does induce an isomorphism
$$L^{\perp} \simeq M^{\perp}$$
compatible with Galois actions.  
As in the proof of Proposition~\ref{prop:ADE}, 
the lattices $L$ and $M$ are isomorphic once we fix an interpretation
via vanishing cycles of our models.

Our assumptions on $p$ guarantee that the classification and deformations of rational double
points over $k$ coincides with the classification in characteristic $0$ \cite{Artin2}.  
Choose new {\em regular} models for $X$ and $Y$
$$\cX'',\cY'' \ra \Spec(R)$$
whose central fibers $\cX''_0$ and $\cY''_0$ are obtained from 
$\tilde{\cX}_0$ and $\tilde{\cY}_0$ by blowing down the $(-2)$-curves
classes associated with $L$ and $M$ respectively.  
Let $\cX_{\circ} \subset \cX''_0$ and $\cY_{\circ} \subset \cY''_0$
denote the smooth loci, i.e., the complements of the rational
curves associated with $L$ and $M$ respectively.  

We claim that $\psi$ induces an isomorphism on
compactly supported cohomology
$$H^2_{c,\text{\'et}}(\bar{\cX}_{\circ},\bQ_{\ell}) \simeq 
H^2_{c,\text{\'et}}(\bar{\cY}_{\circ},\bQ_{\ell}),
$$
compatible with Galois actions.  Indeed, these may be identified with $L^{\perp}$
and $M^{\perp}$, respectively. 
The Weil conjectures yield then that
$$|\cX_{\circ}(k)|=|\cY_{\circ}(k)|$$
and Hensel's Lemma implies our claim.  
\end{proof}

\begin{ques}
Is admitting a model with good or ADE reduction a derived invariant?
\end{ques}
Y.~Matsumoto and C.~Liedtke have recently addressed this.
Having {\em potentially} good reduction is governed by whether
$H^2_{\text{\'et}}(\bar{X},\bQ_\ell)$ is unramified, under
technical assumptions \cite[Th.~1.1]{Mat}.
These are satisified if there exists a Kulikov
model after some basechange \cite[p.~2]{LieMat}.
The ramification condition depends 
on the $\ell$-adic cohomology and thus only on
the derived equivalence class. 
Proposition~\ref{prop:ADE} suggests a monodromy characterization
of ADE reduction in mixed characteristic; under technical assumptions,
there exists such a model if the cohomology is unramified 
\cite[Th.~5.1]{LieMat}.

\section*{Appendix: Semistability and derived equivalence}
To understand the implications of derived equivalence
for rational points over local fields, we must first describe
how monodromy governs the existence of models with good properties.
For K3 surfaces, it is widely expected that unipotent monodromy
should suffice to guarantee the existence of a Kulikov model.
Over $\bC$ this is widely known to experts; the referee
advised us that this was addressed in correspondence
among R.~Friedman, D.~Morrison, and F.~Scattone in 1983.
As we are unaware of a published account, we 
offer an argument:

\begin{theo} \label{theo:Kulikov}
Let $X$ be a K3 surface over $F=\bC((t))$. 
Then $X$ admits a Kulikov model if and only if its monodromy is
unipotent.
\end{theo}
\begin{coro}
Let $X$ and $Y$ be derived equivalent K3 surfaces over $F$.
Then $X$ admits a Kulikov model if and only if $Y$ admits
a Kulikov model. 
\end{coro}
As we have seen, if $X$ and $Y$ are derived equivalent over $F$ then
their Mukai lattices admit a monodromy-equivariant isomorphism; thus the
characteristic polynomials of their monodromy matrices are equal.

The remainder of this section is devoted the proof of Theorem~\ref{theo:Kulikov}.
We start with a review of basic results on Kulikov models.

Let $R=\bC[[t]]$,
$\Delta=\Spec(R)$, and $\Delta^{\circ}=\Spec(F)$.
The monodromy $T$ of $X$ over $\bC((t))$ satisfies
$$(T^e-I)^f=0$$
for some $e,f \in \bN$.
We take $e$ and $f$ minimal with this property.

The semistable reduction theorem \cite{KKMS}
implies there exists an integer $n\ge 1$ such that after basechange to
$$R_2= \bC[[t_2]], F_2=\bC((t_2)), \quad t_2^n=t,$$
there exists a flat proper
$$\pi_2:\cX_2 \ra \Delta_2=\Spec(R_2)$$
such that
\begin{itemize}
\item{the generic fiber is the basechange of $X$ to $F_2$;}
\item{the central fiber $\pi_2^{-1}(0)$ is a reduced normal
crossings divisor.}
\end{itemize}
We call this a {\em semistable model} for $X$.
It is well-known that semistable reductions have unipotent monodromy so
$e|n$.

By work of Kulikov and Persson-Pinkham \cite{Kulikov,PerPin},
there exists a semistable modification of $\cX_2$
$$\varpi:\wcX \ra \Delta_2$$
with trivial canonical class, i.e., there exists a birational
map $\cX_2 \dashrightarrow \wcX$ that is an isomorphism
away from the central fibers.  We call this a {\em Kulikov model} for $X$.
Furthermore, the structure of the central fiber $\wcX_0$ can be
described in more detail:
\begin{enumerate}
\item[Type I]{$\wcX_0$ is a K3 surface and $f=1$.}
\item[Type II]{$\wcX_0$ is a chain of surfaces glued along elliptic
curves, with rational surfaces at the end points and elliptic ruled
surfaces in between; here
$f=2$.}
\item[Type III]{$\wcX_0$ is a union of rational surfaces and $f=3$.}
\end{enumerate}
We will say more about the Type III case:  It determines a combinatorial triangulation of
the sphere
with vertices indexed by irreducible components, edges indexed by double curves,
and `triangles' indexed by triple points \cite{Morrison}.
We analyze this combinatorial structure of $\wcX_0$ in terms of the integer $m$.

Let $\wcX_0=\cup_{i=1}^n V_i$ denote the irreducible components, $\tilde{V}_i$ their normalizations,
and $D_{ij} \subset \tilde{V}_i$ the double curves over $V_i \cap V_j$.
\begin{defi}
$\wcX_0$ is in {\em minus-one} form if for each double curve $D_{ij}$
we have
$({D'_{ij}}^2)_{V_i}=-1$ if $D'_{ij}$ is a smooth component of $D_{ij}$
and $(D_{ij}^2)_{V_i}=1$ if $D_{ij}$ is nodal.
\end{defi}
Miranda-Morrison \cite{MiMo} have shown that
after elementary transformation of $\wcX$, we may assume that $\wcX$ is in minus-one form.

The following are equivalent \cite[\S 3]{Friedman},\cite[0.5,7.1]{FriSca}:
\begin{itemize}
\item{the logarithm of the monodromy is $m$ times a primitive matrix;}
\item{$\wcX_0$ admits a `special $\mu_m$ action', i.e., acting trivially on the
sets of components, double/triple points, and Picard groups of the irreducible
components;}
\item{$\wcX_0$ admits `special $m$-bands of hexagons', i.e., the triangulation
coming from the components of $\wcX_0$ arises as a degree $m$ refinement of
another triangulation.}
\end{itemize}
In other words, $\wcX_0$ `looks like' it is obtained from applying semistable
reduction to the degree $m$ basechange of a Kulikov model.
Its central fiber $\wcX'_0$ can readily be described \cite[4.1]{Friedman}---its
triangulation is the one with refinement equal to the triangulation of $\wcX_0$,
and its components are contractions of the corresponding components of $\wcX_0$.
(When we refer to $\wcX'_0$ below in the Type III case, we mean the
surface defined by this process.)

For Type II we can do something similar \cite[\S 2]{FriedAnnals}.  After elementary modifications,
we may assume the elliptic surfaces are minimal.  Then following are equivalent:
\begin{itemize}
\item{the logarithm of the monodromy is $m$ times a primitive matrix;}
\item{$\wcX_0=V_0\cup_E \ldots \cup_E V_m$ is a chain of $m+1$ surfaces glued along
copies of an elliptic curve $E$, where $V_0$ and
$V_m$ are rational and $V_1,\ldots,V_{m-1}$ are minimal surfaces ruled over $E$.}
\end{itemize}
Again $\wcX_0$ `looks like' it is obtained from applying semistable reduction to
another Kulikov model with central fiber $\wcX'_0=V_0 \cup_E V_m$.
Moreover $(E^2)_{V_0}+(E^2)_{V_m}=0$ and $\wcX'_0$ is $d$-semistable
in the sense of Friedman \cite[2.1]{FriedAnnals}.
(When we refer to $\wcX'_0$ below in the Type II case, we mean the
surface defined by this process.)

There are refined Kulikov models taking into account polarizations:
Let $(X,g)$ be a polarized K3 surface over $F$ of degree $2d$.  Shepherd-Barron \cite{ShB}
has shown there exists a Kulikov
model $\varpi:\wcX \ra \Delta_2$ with the following properties:
\begin{itemize}
\item{there exists a specialization of $g$ to a nef Cartier divisor
on the central fiber $\wcX_0$;}
\item{$g$ is semi-ample relative to $\Delta_2$, inducing
$$\begin{array}{rcccl}
\wcX & & \stackrel{|g|}{\ra} & & \cZ \\
    & \searrow &  & \swarrow &  \\
    &          & \Delta_2    &
\end{array}
$$
where $\wcX_0 \ra \cZ_0$ is birational and $\cZ_0$ has
rational double points, normal crossings, or
singularities with local equations 
$$
xy=zt=0.
$$}
\end{itemize}
These will be called {\em quasi-polarized Kulikov models} and their
central fibers {\em admissible degenerations of degree $2d$.}

Recall the construction in sections five and six of \cite{FriSca}:
Let $\cD$ denote the period domain for degree $2d$ K3 surfaces and $\Gamma$
the corresponding arithmetic group---the orientation-preserving
automorphisms of the cohomology lattice $H^2(X,\bZ)$ fixing $g$.
Fix an admissible degeneration $(\cY_0,g)$ of degree $2d$ and its image $(\cZ_0,h)$,
with deformation spaces $\Def(\cY_0,g)\ra \Def(\cZ_0,h)$; the morphism
arises because $g$ is semiample over the deformation space.
Let
$$\overline{\Gamma \backslash \cD}_{N_{\cY_0}} \supset \Gamma \backslash \cD$$
denote the partial toroidal compactification
parametrizing limiting mixed Hodge structures with monodromy
weight filtration given by a nilpotent $N_{\cY_0}$ associated with
$\cY_0$ (see \cite[p.27]{FriSca}).
We do keep track of the stack structure.  
Given a holomorphic mapping
$$f:\{t:0<|t|<1 \}  \ra \Gamma \backslash \cD,
$$
that is locally liftable (lifting locally to $\cD$), with unipotent monodromy $\Gamma$-conjugate
to $N_{\cY_0}$, then $f$ extends to
$$f:\{t:|t|<1\} \ra \overline{\Gamma \backslash \cD}.$$
The period map extends to an \'etale morphism \cite[5.3.5,6.2]{FriSca}
$$\Def(\cY_0,g) \ra \overline{\Gamma \backslash \cD}.$$
Thus the partial compactification admits a (local) universal family.  

Theorem~\ref{theo:Kulikov} thus boils down to
\begin{quote}
The smallest positive integer $n$ for which we have a Kulikov model
equals the smallest positive integer $e$ such that
$T^e$ is unipotent.
\end{quote}
\begin{proof}
We show that a Kulikov model exists provided
the monodromy is unipotent.
Suppose we have unipotent monodromy over 
$$R_1=\bC[[t_1]], t_1^e=t,$$
and semistable reduction
$$\cX_2 \ra \Delta_2=\Spec(R_2), \quad R_2=\Spec(\bC[[t_2]]), t_2^{me}=t.$$
Let $\wcX \ra \Delta_2$ denote a Kulikov model, obtained after applying elementary transformations
as specified above.  Write
$$mN=\log(T^e)=(T^e-I)-\frac{1}{2}(T^e-I)^2$$
where $m\in \bN$ and $N$ is primitive (cf.\cite[1.2]{FriSca} for the Type III case).

Let $\wcX'_0$ be the candidate for the `replacement' Kulikov model, i.e.,
the central fiber of the Kulikov model we expect to find
$$\wcX' \ra \Delta_1.$$
In the Type I case $\wcX'_0=\wcX_0$ by Torelli, so we focus on the 
Type II and III cases.  

\begin{lemm} \label{lemm:degree2d}
Suppose that $\wcX_0$ admits a degree $2d$ semiample divisor $g$.
Then $\wcX_0'$ admits one as well, denoted by $g'$.
\end{lemm}
\begin{proof}
In the Type II case, this follows from \cite[Th.~2.3]{FriedAnnals}.
The discussion there shows how (after elementary modification) the
divisor can be chosen to induce the morphism $\wcX_0\ra \wcX_0'$
collapsing $V_1\cup_E \ldots \cup_E V_{m-1}$ to $E$, 
interpreted as the normal crossings locus of $\wcX_0'$.

For Type III, we rely on Proposition 4.2 of \cite{Friedman}, which gives an
analogous process for modifying the coefficients of
$h$ so that it is trivial or a sum of fibers
on the special bands of hexagons.  However, Friedman's result
does not indicate whether the resulting line bundle is nef.
This can be achieved after birational modifications of the
total space \cite[Th.~1]{ShB}.
\end{proof}

We can apply the Friedman-Scattone compactification construction
to both $(\wcX_0,g)$ and $(\wcX'_0,g')$, with
$N=N_{\wcX'_0}$ and $mN=N_{\wcX_0}$.  Thus we obtain {\em two} compactifications
$$\overline{\Gamma \backslash \cD}_{mN} \ra
\overline{\Gamma \backslash \cD}_{N}\supset \Gamma \backslash \cD,$$
both with universal families of degree $2d$ K3 surfaces and admissible degenerations.

To construct $\wcX' \ra \Delta_1=\Spec(R_1)$ we use the diagram
$$\begin{array}{ccc}
\Delta_2 & \ra & \overline{\Gamma \backslash \cD}_{mN} \\
\downarrow &     &  \downarrow  \\
\Delta_1  &     & \overline{\Gamma \backslash \cD}_{N}.
\end{array}
$$
The liftability criterion for mappings to the toriodal compactifications
gives an arrow
$$\Delta_1   \ra    \overline{\Gamma \backslash \cD}_{N}$$
making the diagram commute.  
The induced universal family on this space induces a family
$$\wcX'\ra \Delta_1,$$
agreeing with our original family for $t_1\neq 0$ by the Torelli Theorem.
More precisely, the monodromies of the two families over
$\Delta_1^{\circ}=\Spec(R_1)\setminus \{0\}$
are identified and automorphisms of K3 surfaces act faithfully on
cohomology,
so they are isomorphic over $\Delta_1^{\circ}$.
Thus $\wcX'\ra \Delta_1$ is the desired model.
\end{proof}

\bibliographystyle{alpha}
\bibliography{RPDE}

\begin{thebibliography}{KKMSD73}

\bibitem[A'C75]{AC}
Norbert A'Campo.
\newblock La fonction z\^eta d'une monodromie.
\newblock {\em Comment. Math. Helv.}, 50:233--248, 1975.

\bibitem[Add13]{Adding}
Nicolas Addington.
\newblock The {B}rauer group is not a derived invariant.
\newblock To appear in the proceedings of the AIM workshop {\em Brauer groups
  and obstruction problems: moduli spaces and arithmetic}, 2013.
\newblock arXiv:1306.6538.

\bibitem[ADPZ15]{ADPZ}
Kenneth Ascher, Krishna Dasaratha, Alexander Perry, and Rong Zhou.
\newblock Derived equivalences and rational points of twisted {$K3$} surfaces,
  2015.
\newblock arXiv:1506.01374.

\bibitem[AKW14]{AKW}
Benjamin Antieau, Daniel Krashen, and Matthew Ward.
\newblock Derived categories of torsors for abelian schemes, 2014.
\newblock arXiv:1409.2580.

\bibitem[Art74]{Artin}
Michael Artin.
\newblock Algebraic construction of {B}rieskorn's resolutions.
\newblock {\em J. Algebra}, 29:330--348, 1974.

\bibitem[Art77]{Artin2}
Michael Artin.
\newblock Coverings of the rational double points in characteristic {$p$}.
\newblock In {\em Complex analysis and algebraic geometry}, pages 11--22.
  Iwanami Shoten, Tokyo, 1977.

\bibitem[Ati57]{Atiyah}
Michael~F. Atiyah.
\newblock Vector bundles over an elliptic curve.
\newblock {\em Proc. London Math. Soc. (3)}, 7:414--452, 1957.

\bibitem[BM02]{BrMa}
Tom Bridgeland and Antony Maciocia.
\newblock Fourier-{M}ukai transforms for {$K3$} and elliptic fibrations.
\newblock {\em J. Algebraic Geom.}, 11(4):629--657, 2002.

\bibitem[BM14]{BM2}
Arend Bayer and Emanuele Macr{\`{\i}}.
\newblock M{MP} for moduli of sheaves on {K}3s via wall-crossing: nef and
  movable cones, {L}agrangian fibrations.
\newblock {\em Invent. Math.}, 198(3):505--590, 2014.

\bibitem[BMT14]{BM1}
Arend Bayer, Emanuele Macr{\`{\i}}, and Yukinobu Toda.
\newblock Bridgeland stability conditions on threefolds {I}:
  {B}ogomolov-{G}ieseker type inequalities.
\newblock {\em J. Algebraic Geom.}, 23(1):117--163, 2014.
\newblock arXiv:1203.4613.

\bibitem[Bri98]{BrCrelle}
Tom Bridgeland.
\newblock Fourier-{M}ukai transforms for elliptic surfaces.
\newblock {\em J. Reine Angew. Math.}, 498:115--133, 1998.

\bibitem[Bri07]{Br1}
Tom Bridgeland.
\newblock Stability conditions on triangulated categories.
\newblock {\em Ann. of Math. (2)}, 166(2):317--345, 2007.

\bibitem[BV04]{BV}
Arnaud Beauville and Claire Voisin.
\newblock On the {C}how ring of a {$K3$} surface.
\newblock {\em J. Algebraic Geom.}, 13(3):417--426, 2004.

\bibitem[C{\u{a}}l00]{CalThe}
Andrei~Horia C{\u{a}}ld{\u{a}}raru.
\newblock {\em Derived categories of twisted sheaves on {C}alabi-{Y}au
  manifolds}.
\newblock ProQuest LLC, Ann Arbor, MI, 2000.
\newblock Thesis (Ph.D.)--Cornell University.

\bibitem[Fri83]{Friedman}
Robert Friedman.
\newblock Base change, automorphisms, and stable reduction for type {${\rm
  III}\,K3$} surfaces.
\newblock In {\em The birational geometry of degenerations ({C}ambridge,
  {M}ass., 1981)}, volume~29 of {\em Progr. Math.}, pages 277--298.
  Birkh\"auser Boston, Mass., 1983.

\bibitem[Fri84]{FriedAnnals}
Robert Friedman.
\newblock A new proof of the global {T}orelli theorem for {$K3$} surfaces.
\newblock {\em Ann. of Math. (2)}, 120(2):237--269, 1984.

\bibitem[FS85]{FriSca}
Robert Friedman and Francesco Scattone.
\newblock Type {${\rm III}$} degenerations of {$K3$} surfaces.
\newblock {\em Invent. Math.}, 83(1):1--39, 1985.

\bibitem[HLOY04a]{HLOY3}
Shinobu Hosono, Bong~H. Lian, Keiji Oguiso, and Shing-Tung Yau.
\newblock Autoequivalences of derived category of a {$K3$} surface and
  monodromy transformations.
\newblock {\em J. Algebraic Geom.}, 13(3):513--545, 2004.

\bibitem[HLOY04b]{HLOY}
Shinobu Hosono, Bong~H. Lian, Keiji Oguiso, and Shing-Tung Yau.
\newblock Fourier-{M}ukai number of a {K}3 surface.
\newblock In {\em Algebraic structures and moduli spaces}, volume~38 of {\em
  CRM Proc. Lecture Notes}, pages 177--192. Amer. Math. Soc., Providence, RI,
  2004.

\bibitem[Hon13]{Honigs}
Katrina Honigs.
\newblock Derived equivalent surfaces and abelian varieties and their zeta
  functions.
\newblock To appear in the {\em Proceedings of the American Mathematical
  Society}, 2013.
\newblock arXiv:1310.4721.

\bibitem[HS06]{HuSt}
Daniel Huybrechts and Paolo Stellari.
\newblock Proof of {C}\u ald\u araru's conjecture. {A}ppendix: ``{M}oduli
  spaces of twisted sheaves on a projective variety'' [in {\it {m}oduli spaces
  and arithmetic geometry}, 1--30, {M}ath. {S}oc. {J}apan, {T}okyo, 2006; ] by
  {K}. {Y}oshioka.
\newblock In {\em Moduli spaces and arithmetic geometry}, volume~45 of {\em
  Adv. Stud. Pure Math.}, pages 31--42. Math. Soc. Japan, Tokyo, 2006.

\bibitem[Huy08]{HuyJAG08}
Daniel Huybrechts.
\newblock Derived and abelian equivalence of {$K3$} surfaces.
\newblock {\em J. Algebraic Geom.}, 17(2):375--400, 2008.

\bibitem[Huy10]{HuyEMS}
Daniel Huybrechts.
\newblock Chow groups of {K}3 surfaces and spherical objects.
\newblock {\em J. Eur. Math. Soc. (JEMS)}, 12(6):1533--1551, 2010.

\bibitem[Huy12a]{HuyMSRI}
Daniel Huybrechts.
\newblock Chow groups and derived categories of {K}3 surfaces.
\newblock In {\em Current developments in algebraic geometry}, volume~59 of
  {\em Math. Sci. Res. Inst. Publ.}, pages 177--195. Cambridge Univ. Press,
  Cambridge, 2012.

\bibitem[Huy12b]{HuySur}
Daniel Huybrechts.
\newblock A global {T}orelli theorem for hyperk\"ahler manifolds [after {M}.
  {V}erbitsky].
\newblock {\em Ast\'erisque}, (348):Exp. No. 1040, x, 375--403, 2012.
\newblock S{\'e}minaire Bourbaki: Vol. 2010/2011. Expos{\'e}s 1027--1042.

\bibitem[Huy15]{HuyK3}
Daniel Huybrechts.
\newblock Lectures on {$K3$} surfaces, 2015.

\bibitem[HVA13]{HVA2}
Brendan Hassett and Anthony V{\'a}rilly-Alvarado.
\newblock Failure of the {H}asse principle on general {$K3$} surfaces.
\newblock {\em J. Inst. Math. Jussieu}, 12(4):853--877, 2013.

\bibitem[HVAV11]{HVA1}
Brendan Hassett, Anthony V{\'a}rilly-Alvarado, and Patrick Varilly.
\newblock Transcendental obstructions to weak approximation on general {K}3
  surfaces.
\newblock {\em Adv. Math.}, 228(3):1377--1404, 2011.

\bibitem[KKMSD73]{KKMS}
George Kempf, Finn~Faye Knudsen, D.~Mumford, and Bernard Saint-Donat.
\newblock {\em Toroidal embeddings. {I}}.
\newblock Lecture Notes in Mathematics, Vol. 339. Springer-Verlag, Berlin,
  1973.

\bibitem[Kul77]{Kulikov}
Viktor~S. Kulikov.
\newblock Degenerations of {$K3$} surfaces and {E}nriques surfaces.
\newblock {\em Izv. Akad. Nauk SSSR Ser. Mat.}, 41(5):1008--1042, 1199, 1977.

\bibitem[Kul89]{Kuleshov}
Sergej~A. Kuleshov.
\newblock A theorem on the existence of exceptional bundles on surfaces of type
  {$K3$}.
\newblock {\em Izv. Akad. Nauk SSSR Ser. Mat.}, 53(2):363--378, 1989.

\bibitem[Kul90]{Kuleshov2}
Sergej~A. Kuleshov.
\newblock Exceptional bundles on {$K3$} surfaces.
\newblock In {\em Helices and vector bundles}, volume 148 of {\em London Math.
  Soc. Lecture Note Ser.}, pages 105--114. Cambridge Univ. Press, Cambridge,
  1990.

\bibitem[Lan73]{Landman}
Alan Landman.
\newblock On the {P}icard-{L}efschetz transformation for algebraic manifolds
  acquiring general singularities.
\newblock {\em Trans. Amer. Math. Soc.}, 181:89--126, 1973.

\bibitem[LO11]{LO}
Max Lieblich and Martin Olsson.
\newblock {F}ourier-{M}ukai partners of {$K3$} surfaces in positive
  characteristic.
\newblock To appear in the {\em Annales scientifique de l'\'Ecole normale
  sup\'erieure}, 2011.
\newblock arXiv:1112.5114.

\bibitem[Mar10]{Markman2010}
Eyal Markman.
\newblock Integral constraints on the monodromy group of the hyper{K}\"ahler
  resolution of a symmetric product of a {$K3$} surface.
\newblock {\em Internat. J. Math.}, 21(2):169--223, 2010.

\bibitem[Mat14]{Mat}
Yuya Matsumoto.
\newblock Good reduction criteria for {$K3$} surfaces, 2014.
\newblock arXiv:1401.1261.

\bibitem[ML14]{LieMat}
Yuya Matsumoto and Christian Liedtke.
\newblock Good reduction of {$K3$} surfaces, 2014.
\newblock arXiv:1411.4797.

\bibitem[MM83]{MiMo}
Rick Miranda and David~R. Morrison.
\newblock The minus one theorem.
\newblock In {\em The birational geometry of degenerations ({C}ambridge,
  {M}ass., 1981)}, volume~29 of {\em Progr. Math.}, pages 173--259.
  Birkh\"auser Boston, Boston, MA, 1983.

\bibitem[Mor84]{Morrison}
David~R. Morrison.
\newblock The {C}lemens-{S}chmid exact sequence and applications.
\newblock In {\em Topics in transcendental algebraic geometry ({P}rinceton,
  {N}.{J}., 1981/1982)}, volume 106 of {\em Ann. of Math. Stud.}, pages
  101--119. Princeton Univ. Press, Princeton, NJ, 1984.

\bibitem[MSTVA14]{MSTVA}
Kelly McKinnie, Justin Sawon, Sho Tanimoto, and Anthony V\'arilly-Alvarado.
\newblock Brauer groups on {K}3 surfaces and arithmetic applications, 2014.
\newblock arXiv:1404.5460.

\bibitem[Muk87]{Mukai}
Shigeru Mukai.
\newblock On the moduli space of bundles on {$K3$} surfaces. {I}.
\newblock In {\em Vector bundles on algebraic varieties ({B}ombay, 1984)},
  volume~11 of {\em Tata Inst. Fund. Res. Stud. Math.}, pages 341--413. Tata
  Inst. Fund. Res., Bombay, 1987.

\bibitem[Nic11]{Nicaise}
Johannes Nicaise.
\newblock A trace formula for varieties over a discretely valued field.
\newblock {\em J. Reine Angew. Math.}, 650:193--238, 2011.

\bibitem[Nik79]{Nik79}
Viacheslav~V. Nikulin.
\newblock Integer symmetric bilinear forms and some of their geometric
  applications.
\newblock {\em Izv. Akad. Nauk SSSR Ser. Mat.}, 43(1):111--177, 238, 1979.

\bibitem[Nik08]{Nik08}
Viacheslav~V. Nikulin.
\newblock On connected components of moduli of real polarized {$K3$} surfaces.
\newblock {\em Izv. Ross. Akad. Nauk Ser. Mat.}, 72(1):99--122, 2008.

\bibitem[Ogu02]{Oguiso}
Keiji Oguiso.
\newblock K3 surfaces via almost-primes.
\newblock {\em Math. Res. Lett.}, 9(1):47--63, 2002.

\bibitem[Orl97]{Orlov}
Dmitri~O. Orlov.
\newblock Equivalences of derived categories and {$K3$} surfaces.
\newblock {\em J. Math. Sci. (New York)}, 84(5):1361--1381, 1997.
\newblock Algebraic geometry, 7.

\bibitem[PP81]{PerPin}
Ulf Persson and Henry Pinkham.
\newblock Degeneration of surfaces with trivial canonical bundle.
\newblock {\em Ann. of Math. (2)}, 113(1):45--66, 1981.

\bibitem[P{\v{S}}{\v{S}}71]{PSSh}
I.~I. Pjatecki{\u\i}-{\v{S}}apiro and I.~R. {\v{S}}afarevi{\v{c}}.
\newblock Torelli's theorem for algebraic surfaces of type {${\rm K}3$}.
\newblock {\em Izv. Akad. Nauk SSSR Ser. Mat.}, 35:530--572, 1971.

\bibitem[SB83]{ShB}
Nicholas~I. Shepherd-Barron.
\newblock Extending polarizations on families of {$K3$} surfaces.
\newblock In {\em The birational geometry of degenerations ({C}ambridge,
  {M}ass., 1981)}, volume~29 of {\em Progr. Math.}, pages 135--171.
  Birkh\"auser Boston, Mass., 1983.

\bibitem[SD74]{SD}
B.~Saint-Donat.
\newblock Projective models of {$K3$} surfaces.
\newblock {\em Amer. J. Math.}, 96:602--639, 1974.

\bibitem[Sos10]{Sosna}
Pawel Sosna.
\newblock Derived equivalent conjugate {$K3$} surfaces.
\newblock {\em Bull. Lond. Math. Soc.}, 42(6):1065--1072, 2010.

\bibitem[Ste04]{Stellari04}
Paolo Stellari.
\newblock Some remarks about the {FM}-partners of {$K3$} surfaces with {P}icard
  numbers 1 and 2.
\newblock {\em Geom. Dedicata}, 108:1--13, 2004.

\bibitem[Ver13]{Verb}
Misha Verbitsky.
\newblock Mapping class group and a global {T}orelli theorem for hyperk\"ahler
  manifolds.
\newblock {\em Duke Math. J.}, 162(15):2929--2986, 2013.
\newblock Appendix A by Eyal Markman.

\end{thebibliography}

\end{document}